\documentclass[11pt, reqno]{amsart}

\usepackage[utf8]{inputenc}
\usepackage[T1]{fontenc}
\usepackage{amsmath, amssymb, amsthm, mathrsfs}
\usepackage{geometry}
\usepackage{xcolor} 
\usepackage{hyperref}
\usepackage{enumitem}
\usepackage{mathtools}

\hypersetup{
colorlinks=true,
linkcolor=blue,
citecolor=red,
urlcolor=teal
}

\newtheorem{theorem}{Theorem}
\newtheorem{lemma}[theorem]{Lemma}

\newtheorem{conjecture}[theorem]{Conjecture}

\theoremstyle{definition}

\newtheorem{remark}[theorem]{Remark}

\newcommand{\R}{\mathbb{R}}
\newcommand{\Z}{\mathbb{Z}}

\newcommand{\K}{\mathcal{K}}

\newcommand{\opnorm}[2][]{\left\| #2 \right\|_{\mathrm{op}_{#1}}} 

\newcommand{\floor}[1]{\lfloor #1 \rfloor}

\DeclareMathOperator{\dist}{dist}

\title[Quantitative Stability of Betke-Henk-Wills]{Quantitative Stability of the Betke-Henk-Wills Conjecture}

\author{Chao Wang}
\address{School of Future Technology, Shanghai University}
\email{cwang@shu.edu.cn}

\date{\today}

\begin{document}

\begin{abstract}
The Betke-Henk-Wills conjecture provides an upper bound for the lattice point enumerator $G(K, \Lambda)$ of a convex body in terms of its successive minima. While the conjecture is established for orthogonal parallelotopes, its validity for general convex bodies in dimensions $d \ge 5$ remains open. In this paper, we examine the stability of the conjecture under metric perturbations. Specifically, we demonstrate that the inequality is strictly maintained for integer boxes subjected to rotations within a calculated radius, a consequence of the discrete nature of the lattice point enumerator. We derive explicit, geometry-invariant quantitative bounds on the perturbation radius using the operator norm. Furthermore, we extend the analysis to $L_p$-balls for sufficiently large $p$, identifying a sharp threshold $p_0$ for the invariance of the integer hull.
\end{abstract}

\maketitle

\section{Introduction}

The relationship between the enumerative properties of discrete sets and the continuous invariants of convex bodies is a central theme in the Geometry of Numbers \cite{gruber1979geometry, cassels1971introduction}. A fundamental problem in this field concerns the lattice point enumerator $G(K, \Lambda) = |K \cap \Lambda|$ for a convex body $K \subset \R^d$ and a lattice $\Lambda$. While classical results provide volume-based bounds \cite{minkowski1910geometrie}, more refined estimates utilize successive minima or intrinsic volumes \cite{henk2002successive, schneider2014convex}.

Let $\K^d_0$ denote the set of $o$-symmetric convex bodies in $\R^d$, and let $\mathcal{L}^d$ denote the space of full-rank lattices. Minkowski's successive minima $\lambda_i(K, \Lambda)$ are defined as:
\[
\lambda_i(K, \Lambda) = \inf \{ \lambda > 0 : \dim(\lambda K \cap \Lambda) \ge i \}, \quad 1 \le i \le d.
\]

In 1993, Betke, Henk, and Wills proposed a discrete analogue to Minkowski's Second Theorem \cite{betke1993successive}. The conjecture seeks to bound the global lattice point count by the constraints imposed by the successive minima.

\begin{conjecture}[Betke-Henk-Wills] \label{conj:main}
For any $K \in \K^d_0$ and $\Lambda \in \mathcal{L}^d$:
\begin{equation}
G(K, \Lambda) \le \prod_{i=1}^d \left\lfloor \frac{2}{\lambda_i(K, \Lambda)} + 1 \right\rfloor.
\end{equation}
\end{conjecture}

Conjecture~\ref{conj:main} remains unresolved for $d \ge 5$. Existing proofs primarily consider specific symmetries or algebraic properties of the body $K$. In the present work, we shift the focus to local stability. By analyzing the behavior of the integer-valued enumerator relative to the piecewise-constant functional on the right-hand side, we prove that the conjecture is robust under metric deformations. This provides a quantitative characterization of the stability radius for boxes under rotations and $L_p$-deformations.

\section{Preliminaries and Notation}

We consider the Euclidean space $\R^d$ and fix $\Lambda = \Z^d$. Stability is investigated through the action of linear transformations on the body $K$.

\begin{lemma}[Continuity of Successive Minima] \label{lem:lipschitz}
Let $K \in \K^d_0$. For any $T \in GL(d, \R)$, let $\epsilon = \opnorm[K]{T-I}$ and $\epsilon' = \opnorm[K]{T^{-1}-I}$. The successive minima of the transformed body $TK$ satisfy the following bounds \cite{henk2002successive}:
\begin{equation}
    \frac{1}{1 + \epsilon'} \lambda_i(K, \Z^d) \le \lambda_i(TK, \Z^d) \le (1 + \epsilon) \lambda_i(K, \Z^d).
\end{equation}
\end{lemma}

\begin{proof}
Let $\|\cdot\|_K$ denote the gauge function of $K$. For any $x \in \R^d$:
\[ \|Tx\|_K = \|x + (T-I)x\|_K \le \|x\|_K + \opnorm[K]{T-I} \|x\|_K = (1 + \epsilon)\|x\|_K. \]
This inequality implies the inclusion $TK \subseteq (1+\epsilon)K$. By the monotonicity and homogeneity of successive minima:
\[ \lambda_i(TK) \ge \lambda_i((1+\epsilon)K) = \frac{1}{1+\epsilon}\lambda_i(K). \]
Applying the inverse transformation with $\epsilon' = \opnorm[K]{T^{-1}-I}$ yields $T^{-1}K \subseteq (1+\epsilon')K$, which implies $K \subseteq (1+\epsilon')TK$. Consequently:
\[ \lambda_i(K) \ge \lambda_i((1+\epsilon')TK) = \frac{1}{1+\epsilon'}\lambda_i(TK). \]
Rearranging these terms provides the required bounds.
\end{proof}

\section{The Baseline Case: Orthogonal Parallelotopes}

\begin{theorem}[Conjecture for Boxes] \label{thm:box}
Let $K = \{x \in \R^d : |x_i| \le \alpha_i \}$ be an axis-aligned box with semi-axes $\alpha_1 \ge \dots \ge \alpha_d > 0$. Then Conjecture \ref{conj:main} holds.
\end{theorem}

\begin{proof}
For $\Z^d$, the successive minima are $\lambda_i(K, \Z^d) = 1/\alpha_i$. We denote the enumerator by $G(K, \Z^d) = \prod_{i=1}^d (2\floor{\alpha_i} + 1)$ and the upper bound by $\mathcal{R}(K) = \prod_{i=1}^d \floor{2\alpha_i + 1}$. The theorem follows from the elementary inequality $2\floor{x} + 1 \le \floor{2x + 1}$ for $x \ge 0$, as noted in \cite{betke1993successive}.
\end{proof}

\section{Local Stability under Rotation}

\begin{theorem}[Stability in Critical Configurations]
Let $K_0$ be an integer box ($\alpha_i \in \Z$). There exists $\delta > 0$ such that for any $R \in SO(d)$ satisfying $0 < \opnorm{R-I} < \delta$, the inequality $G(K_R, \Z^d) < \mathcal{R}(K_R)$ holds strictly.
\end{theorem}

\begin{proof}
At $R=I$, we have $G(K_0, \Z^d) = \mathcal{R}(K_0) = \prod(2\alpha_i + 1)$. We consider the discrete and continuous components separately.

\textbf{1. Discrete Jumps via Corner Exclusion:} The corners of the integer box $K_0$ are the lattice points $z \in \Z^d$ achieving the maximal Euclidean norm $\|z\|_2 = \sqrt{\sum \alpha_i^2}$. Let $z^*$ be such a corner. For any $R \in SO(d)$, if $z^* \in K_R$, then $R^{-1}z^* \in K_0$. Because $R^{-1}$ is an isometry, $\|R^{-1}z^*\|_2 = \|z^*\|_2$. The only points in $K_0$ achieving this maximal norm are the corners themselves. For a sufficiently small, non-trivial rotation ($\opnorm{R-I} > 0$), the displacement is bounded by $\|R^{-1}z^* - z^*\|_2 \le \opnorm{R^{-1}-I}\|z^*\|_2$. For small $\epsilon$, this displacement is strictly less than the distance between any two distinct corners ($2 \min \alpha_i$). Thus, to remain in $K_0$, $R^{-1}z^*$ must equal $z^*$, implying $z^*$ is an eigenvector of $R$ with eigenvalue 1. Since $R \neq I$ cannot fix all corners of a $d$-dimensional box ($d \ge 2$), there exists at least one corner $z \in \Z^d$ such that $R^{-1}z \notin K_0$, meaning $z \notin K_R$. Consequently, $G(K_R, \Z^d) \le G(K_0, \Z^d) - 1$.

\textbf{2. Rigorous Continuous Stability:} We must guarantee that the functional $\mathcal{R}(K_R)$ does not decrease. We evaluate the gauge norm of the standard basis vectors $e_i$ in the rotated body $K_R$:
\[ \|e_i\|_{K_R} = \|R^{-1}e_i\|_{K_0} = \max_{1 \le j \le d} \frac{|(R^{-1}e_i)_j|}{\alpha_j}. \]
Let $R^{-1}e_i = (\dots, \cos\theta, \dots, \sin\theta, \dots)$. For a sufficiently small perturbation, the $i$-th component is close to 1 (specifically $\le 1$), while off-diagonal components are $O(\epsilon)$. Thus, the maximum is achieved at $j=i$, yielding $\|e_i\|_{K_R} \le 1/\alpha_i = \lambda_i(K_0)$. Because $K_R$ contains $d$ linearly independent lattice vectors $e_1, \dots, e_d$ at scale factors $\le \lambda_i(K_0)$, it strictly follows that $\lambda_i(K_R) \le \lambda_i(K_0)$.
Consequently, $2/\lambda_i(K_R) \ge 2\alpha_i$. Since $2\alpha_i$ is an integer for an integer box, the floor function satisfies:
\[ \left\lfloor \frac{2}{\lambda_i(K_R)} + 1 \right\rfloor \ge 2\alpha_i + 1. \]
Taking the product yields $\mathcal{R}(K_R) \ge \mathcal{R}(K_0)$. Combining this with the discrete drop establishes $G(K_R, \Z^d) < \mathcal{R}(K_R)$.
\end{proof}

\section{Quantitative Stability Bounds}

In this section, we derive a sharp, geometry-invariant radius of stability. By leveraging the isometry property of rotations, we can replace dimension-dependent estimates with a bound based on the body's circumradius.

\begin{theorem}[Refined Stability Radius] \label{thm:stability_radius}
Let $K_0 = \prod_{i=1}^d [-\alpha_i, \alpha_i]$ be an axis-aligned box with $\alpha_1 \ge \dots \ge \alpha_d > 0$. Let $\Delta = \dist(\partial K_0, \Z^d \setminus K_0)$ denote the isolation distance of the lattice points. The Betke-Henk-Wills inequality is strictly maintained for the rotated body $K_R = R K_0$ if the rotation $R \in SO(d)$ satisfies:
\begin{equation}
    \opnorm{R - I} < \frac{\Delta}{\sqrt{\sum_{i=1}^d \alpha_i^2}}.
\end{equation}
\end{theorem}

\begin{proof}
Let $\epsilon = \opnorm{R - I}$. For any $R \in SO(d)$, the operator norm satisfies $\opnorm{R^{-1}-I} = \opnorm{R^T-I} = \epsilon$. To prove that Conjecture~\ref{conj:main} holds, it suffices to show that $G(K_R, \Z^d) \le G(K_0, \Z^d)$ and $\mathcal{R}(K_R) \ge \mathcal{R}(K_0)$.

\textbf{1. Invariance of the Lattice Point Count:}
A lattice point $y \in \Z^d \setminus K_0$ is contained in $K_R$ if and only if $R^{-1}y \in K_0$. By the definition of the isolation distance $\Delta$, for any $y \notin K_0$, we have $\dist(y, K_0) \ge \Delta$. If $R^{-1}y \in K_0$, the following lower bound on the displacement must hold:
\[ \|R^{-1}y - y\|_2 \ge \dist(y, K_0) \ge \Delta. \]

On the other hand, the displacement is bounded from above by the operator norm:
\[ \|(R^{-1} - I)y\|_2 \le \opnorm{R^{-1} - I} \|y\|_2 = \epsilon \|y\|_2. \]

Since $R^{-1}$ is an isometry, for any $y$ such that $R^{-1}y \in K_0$, we have $\|y\|_2 = \|R^{-1}y\|_2 \le R_{\text{diag}}$, where $R_{\text{diag}} = \sqrt{\sum \alpha_i^2}$ is the circumradius of $K_0$. Combining these inequalities:
\[ \Delta \le \epsilon \sqrt{\sum_{i=1}^d \alpha_i^2}. \]
Therefore, if $\epsilon < \Delta / \sqrt{\sum \alpha_i^2}$, no exterior lattice point can enter $K_R$, ensuring $G(K_R, \Z^d) \le G(K_0, \Z^d)$.

\textbf{2. Functional Stability:}
As established in the proof of Theorem 4, for sufficiently small $\epsilon$, the successive minima satisfy $\lambda_i(K_R) \le \lambda_i(K_0)$. This implies $2/\lambda_i(K_R) + 1 \ge 2\alpha_i + 1$, ensuring $\mathcal{R}(K_R) \ge \mathcal{R}(K_0)$. The conjunction of these two conditions confirms the stability of the conjecture.
\end{proof}

\begin{remark}
This refined bound reveals the "curse of dimensionality" in the context of lattice point stability. For a $d$-dimensional unit cube ($\alpha_i = 1/2$, $\Delta = 1/2$), the stability radius scales as $O(d^{-1/2})$. This indicates that as the dimension increases, the range of rotations for which the conjecture is "safely" satisfied for boxes becomes increasingly narrow.
\end{remark}

\section{Asymptotic Stability for $L_p$ Deformations}

Let $K_p(\alpha) = \{ x \in \R^d : \sum_{i=1}^d |x_i/\alpha_i|^p \le 1 \}$. As $p \to \infty$, $K_p$ converges to $K_\infty = \prod [-\alpha_i, \alpha_i]$.

\begin{theorem}[Integer Hull Threshold]
Let $K_\infty = \prod [-\alpha_i, \alpha_i]$. Suppose $\alpha_i \notin \Z$ for all $1 \le i \le d$. There exists an explicit sufficient threshold $p_0$ such that for all $p \ge p_0$, the lattice point set is invariant, $G(K_p, \Z^d) = G(K_\infty, \Z^d)$. This threshold is given by:
\begin{equation}
p_0 = \frac{\ln d}{\min_{i} \ln(\alpha_i / \lfloor \alpha_i \rfloor)}.
\end{equation}
\end{theorem}

\begin{proof}
We first note that the condition $\alpha_i \notin \Z$ is strictly necessary. If any $\alpha_k \in \Z$, the corner point of the integer hull touches the flat face of $K_\infty$ at coordinate $\alpha_k$. For any finite $p$, the strictly convex boundary of the $L_p$ ball immediately excludes this point, making $G(K_p, \Z^d) < G(K_\infty, \Z^d)$ for all $p < \infty$.

Since $K_p \subseteq K_\infty$, we verify the inclusion $K_\infty \cap \Z^d \subseteq K_p$ under the assumption $\alpha_i \notin \Z$. For $z \in K_\infty \cap \Z^d$, we have $|z_i| \le \lfloor \alpha_i \rfloor$. Inclusion in $K_p$ requires:
\[ \sum_{i=1}^d \left( \frac{\lfloor \alpha_i \rfloor}{\alpha_i} \right)^p \le 1. \]
Let $\beta_i = \lfloor \alpha_i \rfloor / \alpha_i$. In the worst case, $d \cdot (\max \beta_i)^p \le 1$. Logarithmic transformation leads to $p \ge \ln d / \ln(1/\beta_{\max})$, which yields the threshold $p_0$.
\end{proof}

\section{Discussion}

The present study establishes that the configurations satisfying the Betke-Henk-Wills conjecture for boxes are not isolated but form stable regions in the space of convex bodies. The quantitative analysis demonstrates that the discrete nature of the lattice point enumerator provides a margin of safety against small metric perturbations. These results suggest that future investigations into the $d \ge 5$ case should focus on bodies in critical contact with the lattice. The explicit stability radius derived herein provides a theoretical foundation for analyzing lattice point counts in the presence of numerical uncertainty.

\bibliographystyle{unsrt}
\bibliography{references}
\end{document}